\documentclass[final]{dmtcs-episciences}

\usepackage{amsmath}
\usepackage{amssymb}
\usepackage{amsfonts}
\usepackage{amsthm}
\usepackage{xcolor}
\usepackage{dashbox}
\usepackage{mathrsfs}
\usepackage{appendix}

\usepackage{tikz}
\usepackage{rotating}
\usepackage[all]{xy}
\usepackage{enumerate}

\newtheorem{theorem}{Theorem}
\newtheorem{lemma}{Lemma}
\newtheorem{corollary}{Corollary}
\newtheorem{proposition}{Proposition}
\theoremstyle{definition}
\newtheorem{definition}{Definition}
\newtheorem{example}{Example}

\theoremstyle{remark}
\newtheorem{remark}{Remark}

\usepackage[utf8]{inputenc}
\usepackage{subfigure}

\usepackage[round]{natbib}

\author{Ioannis Michos\affiliationmark{1}\thanks{Corresponding author.}
  \and Christina Savvidou\affiliationmark{2}}
\title[Enumeration of super-strong Wilf classes]{Enumeration of super-strong Wilf equivalence classes
  of permutations in the generalized factor order}
\affiliation{
  Department of Computer Science and Engineering, European University Cyprus, Nicosia, Cyprus\\
  School of Sciences, UCLan Cyprus, Pyla, Cyprus}
\received{2018-12-22}
\revised{2019-7-18}
\accepted{2019-10-15}
\begin{document}
\publicationdetails{21}{2019}{2}{3}{5055}
\maketitle
\begin{abstract}
Super-strong Wilf equivalence classes of the symmetric group ${\mathcal S}_n$ on $n$ letters, with respect to the generalized factor order, were shown by Hadjiloucas, Michos and Savvidou (2018) to be in bijection with pyramidal sequences of consecutive differences. In this article we enumerate the latter by giving recursive formulae in terms of a two-dimensional analogue of non-interval permutations. As a by-product, we obtain a recursively defined set of representatives of super-strong Wilf equivalence classes in ${\mathcal S}_n$. We also provide a connection between super-strong Wilf equivalence and the geometric notion of shift equivalence---originally defined by Fidler, Glasscock, Miceli, Pantone, and Xu (2018) for words---by showing that an alternate way to characterize super-strong Wilf equivalence for permutations is by keeping only rigid shifts in the definition of shift equivalence. This allows us to fully describe shift equivalence classes for permutations of size $n$ and enumerate them, answering the corresponding problem posed by Fidler, Glasscock, Miceli, Pantone, and Xu (2018).
\end{abstract}

\section{Introduction}
\label{sec:in}
In this work we continue the study of \emph{super-strong Wilf equivalence} on permutations in $n$ letters with respect to the generalized factor order that commenced in \cite{HMS}. S.\ Kitaev et al.\ in \cite{Kitaev} referred to this notion originally as \emph{strong Wilf equivalence}. J.\ Pantone
and V.\ Vatter in \cite{Pantone} used the term ``super-strong Wilf'' to distinguish this from a more general notion they defined and called ``strong Wilf equivalence''.

First we recall some notation and definitions from \cite{Kitaev} and \cite{HMS}. Let $\mathbb P$ and $\mathbb N$ be the sets of positive and non-negative integers, respectively. For each $n, m \in {\mathbb N}$ with $m < n$ we let $[m, n] = \{ m, m+1, \ldots, n \}$ and for each $n \in {\mathbb P}$ we let $[n] = [1, n]$. Let $\mathbb{P}^{*}$  be the set of words on the alphabet $\mathbb{P}$. Its elements are of the form $w= w_1 w_2 \cdots w_i w_{i+1}\cdots w_{n-1} w_n$, with $n \in \mathbb N$ and $w_i \in \mathbb P$ for $i\in [n]$. The {\em reversal} $\tilde{w}$ of $w$ is defined as $\tilde{w} = w_n w_{n-1} \cdots w_{i+1} w_i \cdots w_2 w_1$.
If $n=0$, then $w$ is equal to the {\em empty word}---which will be denoted by $\epsilon$---whereas if $n \in \mathbb P$ and each letter $w_i$ appears exactly once, $w$ is a permutation on $n$ letters. We let $|w|$ be the {\em length} $n$ of the word $w$, $||w||$ be the {\em height} or {\em norm} of $w$ defined as $||w|| = w_{1} + \cdots + w_{i} + \cdots + w_{n}$ and ${\rm{alph}}(w)$ be the set of distinct letters of $\mathbb{P}$ that occur in $w$. The set of all permutations $w$ on $n$ letters such that ${\rm{alph}}(w) = [n]$ is denoted by ${\mathcal S}_n$.

Given $w, u \in \mathbb{P}^{*}$, we say that $u$ is a {\em factor} of $w$ if there exist words $s, v \in \mathbb{P}^{*}$ such that $w = suv$. If $s = \epsilon$ (resp., $v = \epsilon$) $u$ is called a {\em prefix} (resp., {\em suffix}) of $w$. A prefix $u$ of a non-empty word $w$ is called {\em proper}
if $u \neq w$.

Consider the poset $(\mathbb{P}, {\leq})$ with the usual order in $\mathbb{P}$. The {\em generalized factor order} on $\mathbb{P}^{*}$, introduced by Kitaev et al. in \cite{Kitaev}, is the partial order---also denoted by $\leq$ here---obtained by letting $u \leq w$ if and only if there is a factor $v$ of $w$ such that $|u| = |v|$ and $u_{i} \leq v_{i}$, for each $i \in [|u|]$.
Such a factor $v$ of $w$ is called an {\em embedding} of $u$ in $w$. If the first element of $v$ is the $j$-th element of $w$, then the index $j$ is called an {\em embedding index} of $u$ into $w$. The {\em embedding index set} of $u$ into $w$, or \emph{embedding set} for brevity, is defined as the set of all embedding indices of $u$ into $w$ and is denoted by $Em(u, w)$.

For example, if $u = 322$ and $w = 2343213421$, then $u \leq w$ with embedding factors $v = 343$, $v' = 432$ and $v'' = 342$ and corresponding embedding index set $Em(u,w) = \{2, 3, 7\}$. \\

Let now $t, x$ be two commuting indeterminates. The {\em weight} of a word $w \in \mathbb{P}^{*}$ is defined as the monomial $wt(w) = t^{|w|}x^{||w||}$.
A bijection $f:\mathbb{P}^* \rightarrow \mathbb{P}^*$ is called {\em weight-preserving} if the weight of $w$ is preserved under $f$, i.e., $|f(w)| = |w|$ and $||f(w)|| = ||w||$, for every $w \in \mathbb{P}^{*}$.

\begin{definition} (\cite[Section 5]{Kitaev}) \label{def:ssWilf} Two words $u, v \in \mathbb{P}^{*}$ are called {\em super-strongly Wilf equivalent}, denoted $u \, {\sim}_{ss} \, v$, if there exists a weight-preserving bijection $f : \mathbb{P}^{*} \rightarrow \mathbb{P}^{*}$ such that
$Em(u, w) = Em(v, f(w))$, for all $w \in \mathbb{P}^{*}$.
\end{definition}

In \cite{HMS} super-strong Wilf equivalence classes in ${\mathcal S}_n$ were fully characterised using {\em pyramidal sequences of consecutive differences} of permutations. The latter are defined as follows.

\begin{definition}(\cite[Definition 13]{HMS}) \label{def:SeqDiff}
Let $u \in {\mathcal{S}}_{n}$ and $s = s_{1} \cdots s_{i} \cdots s_{n} = u^{-1}$. For $i = n - 1$ down to $1$ consider the suffix
$s_{i} s_{i+1} \cdots s_{n-1} s_{n}$ of $s$ and its alphabet set
${\Sigma}_{i}(s) = {\rm{alph}}(s_{i} s_{i+1} \cdots s_{n-1} s_{n}) = \{ s_{i}^{(i)}, s_{i+1}^{(i)}, \ldots , s_{n-1}^{(i)} , s_{n}^{(i)} \}$,
where $s_{i}^{(i)} < s_{i+1}^{(i)} < \cdots < s_{n-1}^{(i)} < s_{n}^{(i)}$.
We define ${\Delta}_{i}(s)$ to be the vector of \emph{consecutive differences} in ${\Sigma}_{i}(s)$, i.e.,
\[ {\Delta}_{i}(s) = ( s_{i+1}^{(i)} - s_{i}^{(i)}, \ldots , s_{n}^{(i)} - s_{n-1}^{(i)} ). \]
The sequence
\[ p(s) = ( \Delta_1(s), \Delta_2(s), \ldots , \Delta_{n-2}(s), \Delta_{n-1}(s)) \]
has a pyramidal form and is called the {\em pyramidal sequence of consecutive differences} of $s \in {\mathcal S}_n$.
For any $u \in {\mathcal S}_{n}$ note that we always obtain ${\Delta}_1(s) = (\underbrace{1,1, \ldots, 1}_{n-1})$.
\end{definition}

\begin{remark}\label{rem: Delta meaning}
Let $u \in {\mathcal{S}}_{n}$ and $s = u^{-1}$. The vector $\Delta_i(s)$ can also be computed directly from $u$ as the vector of distances between letters in $u$ that are greater than or equal to $i$, as they appear sequentially in $u$ from left to right.
\end{remark}

\vspace{0.1in}

\begin{example}\label{ex:pyramiddiff}
Let $u = 592738164$. Then $s = u^{-1} = 735918462$. The pyramidal sequence of differences for $s$ is the following:
\begin{center}
$\Delta_8(s) = (4)$ \\
$\Delta_7(s) = (2,2)$ \\
$\Delta_6(s) = (2,2,2)$ \\
$\Delta_5(s) = (1,2,2,2)$ \\
$\Delta_4(s) = (1,2,2,2,1)$ \\
$\Delta_3(s) = (1,2,1,1,2,1)$ \\
$\Delta_2(s) = (1,1,1,1,1,2,1)$ \\
$\Delta_1(s) = (1,1,1,1,1,1,1,1)$ \\
\end{center}
\end{example}

The main result of \cite{HMS} is the following.

\begin{theorem}\rm{(\cite[Theorem 15]{HMS}}) \label{seqofdiff}
Let $n \geq 2$ and $u, v \in {\mathcal{S}}_{n}$ with $s = u^{-1}, t = v^{-1}$. Then $u \, {\sim}_{ss} v$ if and only if $p(s) = p(t)$, i.e., if
${\Delta}_{i}(s) = {\Delta}_{i}(t)$, for each $i \in [n-1]$.
\end{theorem}

Note that Theorem \ref{seqofdiff} holds trivially also for $n=1$ and that for every $u \in {\mathcal{S}}_{n}$ we always have
${\Delta}_{i}(u^{-1}) = (\underbrace{1,1, \ldots, 1}_{n-1})$. \\

To enumerate all super-strong Wilf equivalence classes of ${\mathcal S}_{n}$ it suffices to enumerate all distinct pyramidal sequences of consecutive differences of permutations. In Section \ref{sec:pyramida} the latter are constructed explicitly starting always from the vector ${\Delta}_1 = (\underbrace{1,1, \ldots, 1}_{n-1})$ and using three basic rules for the transition from $\Delta_i$ to $\Delta_{i+1}$, namely either by adding up two consecutive entries of the vector $\Delta_i$ or by deleting the leftmost (resp., the rightmost) entry of $\Delta_i$. If $\Delta_i$ is periodic, i.e., of the form $(\underbrace{d,d,\ldots,d}_{n-i})$, for some $d \in \mathbb{P}$, deletion from the left or from the right has the same effect and this makes direct enumeration of pyramidal sequences problematic.

In Section \ref{sec:trapeza} we overcome this difficulty by defining {\em trapezoidal sequences of consecutive differences} as the sequences $({\Delta}_1, \ldots , {\Delta}_{i+1})$ of the first $i+1$ vectors of the corresponding pyramidal structures such that ${\Delta}_{i+1}$ is the only periodic vector, except from $\Delta_1$. It turns out that trapezoidal sequences can be encoded by certain prefixes of length $i$ of permutations in $n$ letters and the latter are enumerated recursively in Section \ref{sec:din} (see Theorem \ref{numbersdin}). Having enumerated trapezoidal sequences of differences it is not so hard to finally enumerate pyramidal sequences of differences and therefore super-strong Wilf equivalences of permutations (see Theorem \ref{sncounting}).
We are also able to construct recursively a set of representatives of super-strong Wilf equivalence classes using the aforementioned prefixes of permutations (see Corollary \ref{reprss}).
Using similar arguments we also recursively enumerate all super-strong Wilf equivalence classes of a given cardinality (see Theorem \ref{sjn}).

Finally, in Section \ref{sec:shift} we establish the connection between super-strong Wilf equivalence and shift equivalence, for permutations. We prove that two permutations are super-strongly Wilf equivalent if and only if they are shift equivalent using only rigid shifts and not reversals (see Proposition \ref{ssrigid}). This allows us to describe fully the shift equivalence classes of permutations (see Theorem \ref{shiftclass}) and finally enumerate them (Corollary \ref{shiftenum}). The latter answers the exact enumeration problem of shift equivalence classes for permutations, posed in \cite[\S 5]{FGMPX}.

\section{Pyramidal sequences of vectors}
\label{sec:pyramida}
Observing the way pyramidal sequences of consecutive differences of a permutation are constructed, we can see that the transition between $\Delta_i$ and $\Delta_{i+1}$ follows one of {\em three simple steps} (see \cite[Lemma 17]{HMS}), namely: add up two consecutive entries of $\Delta_i$; delete the leftmost entry of $\Delta_i$; delete the rightmost entry of $\Delta_i$.

To enumerate such structures it is more convenient to leave aside their connections to permutations and focus on these simple rules that can indeed construct all possible pyramidal sequences of this type.

\begin{definition}\label{def:pyramidal}
A {\em pyramidal sequence of vectors} is a sequence of the form
\[ p = ({\Delta}_1, \ldots , {\Delta}_i, {\Delta}_{i+1}, \ldots , {\Delta}_{n-1}), \]
where each ${\Delta}_i$ is a sequence of $n-i$ positive integers such that ${\Delta}_1 = (\underbrace{1,1, \ldots, 1}_{n-1})$ and if
${\Delta}_{i} = (d_1, d_2, \ldots , d_{n-i-1}, d_{n-i})$ we have the following three options for $\Delta_{i+1}$:
\[ {\Delta}_{i+1} = \begin{cases}
(d_1,  \ldots, d_{k-1}, d_k + d_{k+1}, d_{k+2}, \ldots, d_{n-i}), \, \text{for some} \, k \in [n-i-1], \text{or} \\
(d_2,  \ldots, d_{n-i-1}, d_{n-i}), \quad \text{or} \\
(d_1, d_2,  \ldots,  d_{n-i-1}). \\
\end{cases} \]
\end{definition}
It is important to note that if $\Delta_i = (\underbrace{d, d, \ldots, d}_{n-i})$ for some $d\in \mathbb{P}$, the second and third options coincide.

Let $\varPi_n$ denote the set of all pyramidal sequences of the above form.

\begin{lemma}
Given a permutation $u \in {\mathcal S}_n$, its pyramidal sequence $p(u^{-1})$ lies in $\varPi_n$.
\end{lemma}

\begin{proof}
Following the notation of Definition \ref{def:SeqDiff} there are three distinct cases to consider for the transition from ${\Delta}_{i}(s)$ to ${\Delta}_{i-1}(s)$.
If $s_{i-1} < s_{i}^{(i)}$, then the difference $s_{i}^{(i)} - s_{i-1}$ will be added at the left of ${\Delta}_{i}(s)$. If $s_{i-1} > s_{n}^{(i)}$, then the difference $s_{i-1} - s_{n}^{(i)}$ will be added at the right of ${\Delta}_{i}(s)$.
Finally, if there exists an index $k$ such that $s_{k}^{(i)} < s_{i-1} < s_{k+1}^{(i)}$, then the difference $s_{k+1}^{(i)} - s_{k}^{(i)}$ in ${\Delta}_{i}(s)$ will break up into the two summands $s_{i-1} - s_{k}^{(i)}$ and $s_{k+1}^{(i)} - s_{i-1}$ in ${\Delta}_{i-1}(s)$.
\end{proof}

For the converse, we have the following construction.

\begin{lemma}\label{lem: construction}
Given an element $p \in \varPi_n$, there exists  $u \in {\mathcal S}_n$ such that $p = p(u^{-1})$.
\end{lemma}

\begin{proof}
In view of Remark \ref{rem: Delta meaning}, we construct such a permutation recursively, using the connection between $\Delta_{i+1}$ and $\Delta_{i}$ in each step. We start from letters $n$ and $n-1$ and place them in distance $d$, where $\Delta_{n-1} = (d)$. There are two options for this. We choose the one with $n-1$ on the left.

Let ${\Delta}_{i} = (d_1, d_2, \ldots , d_{n-i-1}, d_{n-i})$ and suppose, by induction, that all letters greater than $i$ are placed in the word and their distances agree with ${\Delta}_{i+1}$. In view of Definition \ref{def:pyramidal}, there are three options for the transition from $\Delta_i$ to $\Delta_{i+1}$. If the $k$-th coordinate in $\Delta_{i+1}$ is equal to $d_{k} + d_{k+1}$, then letter $i$ is placed in distance
$d_1 + d_2 + \cdots + d_{k-1} + d_k$ to the right of the leftmost letter of the previous step. If $\Delta_{i+1} = (d_2, \ldots, d_{n-i})$
(resp., $\Delta_{i+1} = (d_1, \ldots, d_{n-i-1})$), then letter $i$ is placed in distance $d_1$ (resp., $d_{n-i}$) to the left of the leftmost (resp., to the right of the rightmost) letter of the previous step.

Throughout this step-by-step construction, we have either one or two choices for the placement of each new letter. The latter occurs only on transitions from $\Delta_i = (\underbrace{d, d, \ldots, d}_{n-i})$ to $\Delta_{i+1} = (\underbrace{d, d, \ldots, d}_{n-i-1})$. In this case we choose again to place  letter $i$ at distance $d$ to the left of the leftmost letter of the previous step.
\end{proof}

\begin{remark}\label{rem:2power}
The above proof is an alternative way to show that the cardinality of each super-strong Wilf equivalence class of permutations is equal to $2^j$ for some $j$. This had been shown with a more technical and difficult proof using the notion of {\em cross equivalence} in \cite[Corollary 27]{HMS}. Moreover, it is evident that the exponent $j$ is equal to the number of transitions from $\Delta_i = (\underbrace{d, d, \ldots, d}_{n-i})$ to $\Delta_{i+1} = (\underbrace{d, d, \ldots, d}_{n-i-1})$, for $i\in [n-1]$. Note that by convention $\Delta_n$ is the empty vector, so there is always such a transition from $\Delta_{n-1}$ to $\Delta_n$.
\end{remark}

We demonstrate the above construction by way of an example.

\begin{example} (cf. \cite[Example 18]{HMS}) \\
Consider the pyramidal sequence of vectors appearing in Example \ref{ex:pyramiddiff}. Following the steps described in Lemma \ref{lem: construction}, one possible construction for $u$ is shown in the following diagram, where the symbol $\circ$ is used to represent smaller letters than the ones appearing in every step---which are shown inside a box here---and each sequence below an under brace denotes the corresponding vector $\Delta_i$.
\[ \underbrace{\fbox{8}\ \circ \ \circ \ \circ \ 9}_4 \rightarrow  \underbrace{8\ \circ \ \fbox{7} \ \circ \ 9\ }_{\ \ 2 \  \ \ \ \ \ \ 2}  \rightarrow \underbrace{\fbox{6} \ \circ \ 8 \ \circ \ 7\ \circ 9}_{2 \ \ \ \ \ 2 \ \ \ \ \ 2} \rightarrow \underbrace{\fbox{5} \ 6\ \circ \ 8 \ \circ \ 7\ \circ 9}_{1 \ \ \ 2 \ \ \ \ \ 2 \ \ \ \ \ 2} \]
 \[\rightarrow \underbrace{5 \ 6\ \circ \ 8 \ \circ \ 7\ \circ 9 \ \fbox{4}}_{1 \ \ \ 2 \ \ \ \ \ 2 \ \ \ \ \ 2 \ \ \ 1} \rightarrow \underbrace{5 \ 6\ \circ \ 8 \ \fbox{3} \ 7\ \circ 9 \ 4}_{1 \ \ \ 2 \ \ \ 1 \ \ 1 \ \ \ 2 \ \ \ 1} \rightarrow \underbrace{5 \ 6\ \fbox{2} \ 8 \ 3 \ 7\ \circ \ 9 \ 4}_{1 \ \ 1 \ \ 1 \ \ 1 \ \ 1 \ \ \ 2 \ \ \ 1} \] \[ \rightarrow \underbrace{5 \ 6\ 2 \ 8 \ 3 \ 7\ \fbox{1} \ 9 \ 4}_{1  \ 1 \ \ 1  \ 1 \ \ 1 \ 1 \ \ 1 \  1} = u.\]

\end{example}
\vspace{0.2cm}

\begin{remark}\label{rem:balls-walls}
It is helpful to view Definition \ref{def:pyramidal} in the following way. Suppose that we originally have $n$ walls which define $n-1$ chambers with one ball in each one of them. This is precisely the situation in ${\Delta}_1$. Then at each step, the transition from ${\Delta}_i$ to ${\Delta}_{i+1}$  can be visualized by the removal of some wall. If this wall is internal, the balls at its left and right chambers will all be concentrated at one unified chamber. On the other hand, if this wall is external, all corresponding balls to its left (if it is a right wall) or to its right (if it is a left one) will be removed. This combinatorial game ends when only two walls, i.e., one chamber, are left. We want to enumerate the number of ways that this can be done, considering that two moves are different if they result to a different set-up of chambers and balls.
\end{remark}

For technical reasons, we extend the notion of consecutive differences for permutations to number sets. More precisely, for a set $X = \{x_1 < x_2 < \cdots < x_{k-1} < x_k\} \subset \mathbb{P}$, let
\[ \Delta(X) = (x_2 - x_1, \ldots, x_k - x_{k-1}) \]
be the vector of consecutive differences in $X$.
For a given vector $\Delta = (d_1, \ldots , d_{k-1})$ and a given $c \in \mathbb{P}$ we also set
\[ \overline{c}(\Delta) = \{ c, c + d_1, \ldots , c + d_1 + \cdots + d_{k-1} \}. \]
Note that, given a set $X$ and a vector $\Delta$, we have that $\Delta = \Delta(X)$ if and only if $X=\overline{c}(\Delta)$, for some $c \in \mathbb{P}$.

\section{Trapezoidal sequences of consecutive differences}
\label{sec:trapeza}
Let $l \geq 2$. A permutation $w$ of length $l$ (resp., a set $A$ of cardinality $l$) is called {\em periodic} when the vector of consecutive differences of ${\rm{alph}}(w)$ (resp., of $A$) is equal to $(\underbrace{d,d,\ldots,d}_{l-1})$, for some $d \in \mathbb{P}$.

\noindent A vector $\Delta$ is called \emph{periodic} if $\Delta = (d,d,\ldots,d)$, for some $d\in\mathbb{P}$.

\begin{definition}\label{def:lower pyramids}
For a fixed $i \in [n-2]$, a {\em trapezoidal sequence of consecutive differences of height $i$} is a sequence
$( {\Delta}_{1}, \ldots , {\Delta}_{i}, {\Delta}_{i+1} )$ of the $i+1$ initial parts of a pyramidal sequence in $\varPi_n$ such that:
\begin{itemize}
\item $\Delta_{i+1}$ is periodic, and
\item $\Delta_j$ is not periodic, for all $j\in [2,i]$.
\end{itemize}
Let $\varDelta_{i,n}$ denote the set of all trapezoidal sequences of height $i$, inherited by pyramidal sequences of $\varPi_n$.
\end{definition}

\begin{example}\label{ex:lowerpyramids}
The sequence of differences
\begin{center}
$(2,2,2)$ \\
$(1,2,2,2)$ \\
$(1,2,2,2,1)$ \\
$(1,2,1,1,2,1)$ \\
$(1,1,1,1,1,2,1)$ \\
$(1,1,1,1,1,1,1,1)$ \\
\end{center}
lies in $\varDelta_{5,9}$; it is the sequence of the $6$ initial parts of the pyramidal sequence in Example \ref{ex:pyramiddiff}.
\end{example}

We will encode trapezoidal sequences of consecutive differences of height $i$ using a certain type words of length $i$, defined as follows.

\begin{definition}\label{def:prefixes}
Let $n \geq 3$. For $i \in [n-2]$ we define the set ${\mathcal D}_{i,n}$ as the set of all words $u = u_1 u_2 \cdots u_i$ of length $i$ such that:
\begin{itemize}
\item $u$ appears as a non-empty prefix of a permutation $w$ in $\mathcal{S}_n$, whose remaining $(n-i)$-letter suffix is periodic, i.e., the set $[n] \setminus {\rm{alph}}(u)$ is periodic, and
\item $u$ is the prefix of smallest length with the above property, i.e., for all $j < i$ the set $[n] \setminus \{ u_1 , u_2 , \ldots , u_j \}$ is not periodic.
\end{itemize}
\end{definition}

\begin{example}
The word $u = 73591$ lies in ${\mathcal D}_{5,9}$. Indeed, the set $[9] \setminus {\rm{alph}}(u) = [9] \setminus \{ 1, 3, 5, 7, 9 \} = \{2, 4, 6, 8 \}$ is periodic since its vector of consecutive differences is $(2,2,2)$. On the other hand, for each of the proper prefixes $u' \in \{7359, 735, 73, 7 \}$ of $u$, it is easy to check that the set $[9] \setminus {\rm{alph}}(u')$ is not periodic.
\end{example}

\vspace{0.1in}

The following result yields a recursive algorithm for the construction of sets ${\mathcal D}_{i,n}$.

\begin{proposition}
For $n=3$ we have ${\mathcal D}_{1,3} = \{1, 2, 3\}$. For $n \geq 4$ we construct ${\mathcal D}_{i,n}$ recursively as follows:
\begin{enumerate}
\item For $i = 1$ we have ${\mathcal D}_{1,n} = \{1, n\}$.
\item For $i \geq 2$, ${\mathcal D}_{i,n}$ is the set of prefixes $u$ of length $i$ of permutations in $\mathcal{S}_n$ whose corresponding suffix is periodic and for each proper prefix $u'$ of $u$ of length $1\leq j<i$ we have $u' \notin {\mathcal D}_{j,n}$.
\end{enumerate}
\end{proposition}

\begin{proof}
For $n=3$ the result is evident. For $n \geq 4$ and $i=1$ it follows immediately. Let $i > 1$ and $u \in {\mathcal D}_{i,n}$. Then each proper prefix $u'$ of $u$ has the property that
$[n] \setminus {\rm{alph}}(u')$ is not periodic, therefore $u' \notin {\mathcal D}_{j,n}$, for $j = |u'|$.

For the reverse direction suppose that $u$ is a prefix of some permutation on $n$ letters such that $[n] \setminus {\rm{alph}}(u)$ is periodic and for each proper prefix $u'$ of $u$ we have $u' \notin {\mathcal D}_{j,n}$, for $j = |u'|$. Then we must show that $[n] \setminus {\rm{alph}}(u')$ is not periodic. If $j=1$ then $u' \notin \{1, n \}$ and the result follows. Thus we may suppose that $j > 1$ and for the sake of contradiction let us assume that
$[n] \setminus {\rm{alph}}(u')$ is periodic. Consider the set
\[ L = \{ |t| \: : \: t \: {\rm{proper \, prefix \,of}} \: u \quad {\rm{and}} \quad [n] \setminus {\rm{alph}}(t) \:  {\rm{is \: periodic}} \}. \]
Clearly $L \neq \emptyset$, as $|u'| \in L$. By the well-ordering principle, $\min(L)$ exists. Set $\min(L) = l$. Then $l > 1$ and there exists a proper prefix $x$ of $u$ of length $l$ such that $[n] \setminus {\rm{alph}}(x)$ is periodic. Furthermore, for each proper prefix $y$ of $x$ we get
$[n] \setminus {\rm{alph}}(y)$ is not periodic. Then by definition $x \in {\mathcal D}_{l,n}$. On the other hand, as $x$ is a proper prefix of $u$, by our assumption $x \notin {\mathcal D}_{l,n}$, a contradiction.
\end{proof}

\begin{example}
For $n=5$ by definition we get ${\mathcal D}_{1,5} = \{1,5\}$. To calculate the set of prefixes ${\mathcal D}_{2,5}$ observe that the only possible vectors of differences for the corresponding periodic sets in this case are either $(1,1)$ or $(2,2)$. These correspond to the sets $\{1,2,3\}, \{2,3,4\}, \{3,4,5\}$ and $\{1,3,5\}$, respectively. All possible prefixes with letters in the complements of the above sets are $45, 54; 15, 51; 12, 21$ and $24, 42$, respectively. The prefixes $54, 15, 51,$ and $12$ are rejected since they have a proper prefix in ${\mathcal D}_{1,5}$. Hence, we obtain that
${\mathcal D}_{2,5} = \{21, 24, 42, 45\}$.
\end{example}

\vspace{0.2cm}

We are now ready to show that each trapezoidal sequence in $\varDelta_{i,n}$ can be encoded as a word in ${\mathcal D}_{i,n}$. The main idea is the following. Each word in ${\mathcal D}_{i,n}$ may be viewed as a sequence of $i$ walls that are going to be deleted one after the other, starting from the original $n$ walls containing $n-1$ balls and ending up with an equal number of balls in each chamber, a set-up described in Remark \ref{rem:balls-walls}. Taking into consideration the way that balls are separated after each wall deletion, a unique trapezoidal sequence in $\varDelta_{i,n}$ is constructed bottom-up. This turns out to be a well defined map which we will show to be a bijection.

\begin{proposition}\label{prop:prefix<->trapezoid}
There is a bijection between the set of prefixes ${\mathcal D}_{i,n}$ and the set of trapezoidal sequences of  consecutive differences $\varDelta_{i,n}$.
\end{proposition}
\begin{proof}
Suppose that $u = u_1 \cdots u_i \in {\mathcal D}_{i,n}$. We construct a unique element $p \in \varDelta_{i,n}$ in the following way.
First, we define sets $X_j$, $j \in [i+1]$, inductively as
\begin{equation} \label{eq: uXp}
X_1 = [n] \text{ and } X_{j+1} = X_j \setminus \{u_j\}, \text{ for } j \in [i].
\end{equation}
The image of $u$ is then defined to be
\[\phi(u) = p = (\Delta_1, \ldots, \Delta_j, \ldots, \Delta_{i+1}), \text{ where } \Delta_j = \Delta(X_j).\]

For the reverse direction, let $p = (\Delta_1, \ldots, \Delta_j, \Delta_{j+1}, \ldots, \Delta_{i+1}) \in \varDelta_{i,n}$.
Suppose that $\Delta_{j+1} = (d_1, \ldots, d_{n-j-1})$ and $\Delta_{j} = (e_0, e_1,  \ldots, e_{n-j-1})$.
We define sets $Y_j$ inductively as follows. Set $Y_1 = [n]$. For $j \in [2, i]$, let $y_j = \min (Y_j)$ and set
\begin{equation} \label{eq: pYu}
 Y_{j+1} = \begin{cases}
\overline{y_j}(\Delta_{j+1}), \text{ if } \Delta_{j+1} \neq  ( e_1,  \ldots, e_{n-j-1})\\
\overline{y_j + e_0}(\Delta_{j+1}), \text{ if } \Delta_{j+1} =  (  e_1,  \ldots, e_{n-j-1}).\\
\end{cases}
\end{equation}
We then set $\psi(p) = u = u_1 \cdots u_j \cdots u_i$, where $\{u_j\} = Y_{j+1} \setminus Y_j$.

By construction and in view of Definitions \ref{def:lower pyramids} and \ref{def:prefixes}, we obtain that $\phi(u) \in \varDelta_{i,n}$ and
$\psi(p) \in {\mathcal D}_{i,n}$.
\vspace{0.2cm}

It remains to show that $\phi$ and $\psi$ are inverses of each other, i.e., ${\psi}({\phi}(u)) = u$ and ${\phi}({\psi}(p)) = p$,
for each $u \in {\mathcal D}_{i,n}$ and $p \in \varDelta_{i,n}$.

For the former equality it suffices to demonstrate inductively on $j$ that $Y_j = X_j$, given the set $X_j$ in \eqref{eq: uXp}, for $j \in [i]$ and the recursive construction of the sets $Y_j$ in \eqref{eq: pYu}.
For $j=1$ we immediately get $Y_1 = X_1 = [n]$. Assuming that $Y_k = X_k$, and thus $y_k = \min(Y_k) = \min(X_k) = x_k$, we have to show that
$Y_{k+1} = X_{k+1}$.
If $\Delta_{k+1} \neq  ( e_1,  \ldots, e_{n-k-1})$, then $Y_{k+1} = \overline{y_k}(\Delta_{k+1}) = \overline{x_k}(\Delta_{k+1})$.
Now since ${\Delta}_{k+1} = {\Delta}(X_{k+1})$ we get $\overline{x_k}(\Delta_{k+1})= X_{k+1}$ and the result follows.
If, on the other hand, $\Delta_{k+1} =  ( e_1,  \ldots, e_{n-k-1})$, then we obtain that $Y_{k+1} = \overline{y_k + e_0}(\Delta_{k+1}) = \overline{x_k + e_0}(\Delta_{k+1})$. Clearly $y_k \notin Y_{k+1}$ and in fact we have $\{ y_k \} = Y_k \setminus Y_{k+1}$ and furthermore $y_k = u_k$.
Since $X_{k+1} = X_k \setminus \{u_k \}$, we obtain $X_{k+1} = Y_{k+1}$.

For the latter equality, our arguments are reversed starting with the given set $Y_j$ in \eqref{eq: pYu} and using the recursive construction of the sets $X_j$ in \eqref{eq: uXp}.
\end{proof}

\begin{example} The image of the word $u = 7 3 5 9 1 \in \mathcal{D}_{5,9}$ via the bijection $\phi$ is the trapezoidal sequence of differences appearing bottom-up on the right column of the following table.
\[
\begin{array}{c c l c c}
&  &   |_{1} \bullet |_2 \bullet |_3  \bullet |_4  \bullet |_5  \bullet |_6 \bullet |_7 \bullet |_8 \bullet |_9  & \rightarrow & \Delta_1 = (1, 1, 1, 1, 1, 1, 1, 1) \\
\color{red} 7 & \rightarrow  & |_{1} \bullet |_2 \bullet |_3  \bullet |_4  \bullet |_5  \bullet |_6 \bullet \hspace{0.1cm} { }_{\color{red} 7} \bullet |_8 \bullet |_9  & \rightarrow & \Delta_2 = (1,1,1,1,1,2,1)  \\
\color{red}  3  & \rightarrow  & |_{1} \bullet |_2 \bullet \hspace{0.1cm} { }_{\color{red} 3}  \bullet |_4  \bullet |_5  \bullet |_6 \bullet \hspace{0.1cm} { }_{\color{red} 7} \bullet |_8 \bullet |_9  & \rightarrow & \Delta_3 = (1,2,1,1,2,1) \\
 \color{red}  5  & \rightarrow & |_{1} \bullet |_2 \bullet \hspace{0.1cm} { }_{\color{red} 3}  \bullet |_4  \bullet \hspace{0.1cm} { }_{\color{red} 5}  \bullet |_6 \bullet \hspace{0.1cm} { }_{\color{red} 7} \bullet |_8 \bullet |_9  & \rightarrow & \Delta_4 = (1,2,2,2,1) \\
\color{red} 9  & \rightarrow & |_{1} \bullet |_2 \bullet \hspace{0.1cm} { }_{\color{red} 3}  \bullet |_4  \bullet \hspace{0.1cm} { }_{\color{red} 5}  \bullet |_6 \bullet \hspace{0.1cm} { }_{\color{red} 7} \bullet |_8 \hspace{0.4cm} { }_{\color{red} 9}  & \rightarrow & \Delta_5 = (1,2,2,2)  \\
\color{red} 1  & \rightarrow & \hspace{0.1cm} { }_{\color{red} 1} \hspace{0.35cm} |_2 \bullet \hspace{0.1cm} { }_{\color{red} 3}  \bullet |_4  \bullet \hspace{0.1cm} { }_{\color{red} 5}  \bullet |_6 \bullet \hspace{0.1cm} { }_{\color{red} 7} \bullet |_8 \hspace{0.4cm} { }_{\color{red} 9}  & \rightarrow & {\Delta}_{6} = (2, \, 2 , \, 2)  \\
\end{array}
\]
\end{example}

Now set $d_{i,n} = |{\mathcal D}_{i,n}| = |\varDelta_{i,n}|$ and let $s_{n}$ be the number of distinct super-strong Wilf equivalence classes of ${\mathcal S}_{n}$. It is trivial to see that $s_1 = s_2 = 1$ and $s_3 = 2$. For $n \geq 4$ the enumeration of super-strong Wilf equivalence classes is achieved using the numbers $d_{i,n}$ as follows.

\begin{theorem} \label{sncounting}
Let $n \geq 4$. The number $s_n$ of distinct super-strong Wilf equivalence classes of ${\mathcal S}_{n}$ is given by the recursive formula
\begin{equation} \label{eq: sncount}
 s_n =  s_{n-1} + \sum_{i=2}^{n-2} d_{i,n}\cdot s_{n-i}.
\end{equation}
\end{theorem}

\begin{proof}
Let ${\mathcal T}_{i,n} = \{ (\Delta_1, \ldots , \Delta_i, \Delta_{i+1}, \ldots, \Delta_{n-1}) \in {\varPi}_{n} \, :
                                \, (\Delta_1, \ldots , \Delta_i, \Delta_{i+1}) \in \varDelta_{i,n} \}$, for $i \in [n-2]$.
We clearly have
\begin{equation} \label{eq:partition} {\varPi}_{n} =  {\mathcal T}_{1,n} \sqcup {\mathcal T}_{2,n} \sqcup \cdots \sqcup {\mathcal T}_{i,n} \sqcup
                                                                                     \cdots \sqcup {\mathcal T}_{n-2,n}. \end{equation}
Observe that ${\varDelta}_{1,n}$ consists of just one element, namely $(\Delta_1, \Delta_2 )$, where $\Delta_1 = (\underbrace{1,1,\ldots,1}_{n-1})$ and
$\Delta_2 = (\underbrace{1,1,\ldots,1}_{n-2})$. Then there is an immediate bijective correspondence between $\varPi_{n-1}$ and ${\mathcal T}_{1,n}$, therefore $|{\mathcal T}_{1,n}| = s_{n-1}$.

Now let $i \in [2, n-2]$. Consider a pyramidal sequence in ${\mathcal T}_{i,n}$.
Then there exists a $d \in [n-1]$ such that $\Delta_{i+1} = (\underbrace{d,d,\ldots,d}_{n-i-1})$ and for all $j \in [2, i]$ there exists no $e$ such that
$\Delta_{j} = (\underbrace{e,e,\ldots,e}_{n-j})$.

Our first claim is that all entries in $\Delta_k$ for $k \in [i+1, n-1]$ will be multiples of $d$. Indeed, suppose that $\Delta_k = (d_1, \ldots , d_{n-k})$, where by induction it is assumed that $d \, | \, d_i$, for all $l \in [n-k]$. If $\Delta_{k+1}$ is equal either to $(d_2, \ldots , d_{n-k})$ or $(d_1, \ldots , d_{n-k-1})$, then the result follows immediately. On the other hand, if $\Delta_{k+1} = (d_{1} , \ldots , d_{m-1}, d_{m} + d_{m+1}, d_{m+2}, \ldots , d_{n-k})$ it is enough to show that $d \, | \, (d_{m} + d_{m+1})$, which follows inductively from $d \, | \, d_{m}$ and $d \, | \, d_{m+1}$. \\

We now define a map ${\tau}_{i,n} : {\mathcal T}_{i,n} \rightarrow \varDelta_{i,n} \times \varPi_{n-i}$ as
{\footnotesize
\begin{equation} \label{eq:bijection_t} {\tau}_{i,n}( \Delta_{1}, \ldots , \Delta_{i}, \Delta_{i+1}, \ldots , \Delta_{n-1}) = \big( (\Delta_{1}, \ldots , \Delta_{i}, \Delta_{i+1}),
\frac{1}{d} \cdot (\Delta_{i+1}, \ldots , \Delta_{n-1}) \big) , \end{equation}}
where $\Delta_{i+1} =  (\underbrace{d,d,\ldots,d}_{n-i-1})$. Our previous claim ensures that ${\tau}_{i,n}$ is well defined. Furthermore, it is also a bijection whose inverse is the map $\rho_{i, n} : \varDelta_{i,n} \times \varPi_{n-i}
\rightarrow {\mathcal T}_{i,n}$ defined as
{\footnotesize
\[ \rho_{i,n} \big( (\Delta_{1}, \ldots , \Delta_{i}, \Delta_{i+1}), ({\Delta}^{'}_1, \ldots , {\Delta}^{'}_{n-i-1}) \big) =
(\Delta_1, \ldots , \Delta_i, \Delta_{i+1}, d \cdot {\Delta}^{'}_1, \ldots , d \cdot {\Delta}^{'}_{n-i-1}), \] }
\vspace{-0.2cm}

\noindent where $\Delta_{i+1} =  (\underbrace{d,\ldots,d}_{n-i-1})$. By direct enumeration, it follows that $|{\mathcal T}_{i,n}| = d_{i,n} \cdot s_{n-i}$, for $i \in [2, n-2]$.
\end{proof}

\vspace{0.2cm}

Our next goal is to find a set of representatives of super-strong Wilf equivalence classes in $\mathcal{S}_n$. To do this, we start with the notion of {\em{the reduced form of a permutation}}. For a permutation $\tau$, its reduced form ${\rm{red}}(\tau)$ is the permutation obtained by replacing the smallest entry of $\tau$ by $1$, the second smallest by $2$, and so on.

For technical reasons we define ${\mathcal E}_{i,n}$ to be
\[ {\mathcal E}_{i,n} = \begin{cases}
\{ 1 \}, \quad \mbox{if} \quad i=1 \\
{\mathcal D}_{i,n}, \quad \mbox{if} \quad i \in [2, n-2]. \\
\end{cases} \]

\begin{corollary} \label{reprss}
Let ${\mathcal R}_{n}$ be the set of permutations in ${\mathcal S}_{n}$ defined recursively as
\[ {\mathcal R}_n = \{ u \cdot v \: : \: u \in {\mathcal E}_{i,n}; \,\;  {\rm{red}}(v) \in {\mathcal R}_{n-i}; \,\; i \in [n-2] \}. \]
Then a set of super-strong Wilf equivalence classes representatives in ${\mathcal S}_n$ is given by the set
\[ {\mathcal C}_n = \{ w^{-1} \: : \:  w \in {\mathcal R}_{n} \}. \]
\end{corollary}

\begin{proof}
In view of Theorem \ref{sncounting}, the cardinality of ${\mathcal R}_n$ is the correct one. To construct a permutation $u\cdot v \in \mathcal{R}_n$, we fix a prefix $u \in {\mathcal E}_{i,n}$ and a permutation $\tau \in \mathcal{R}_{n-i}$. Then there is a unique suffix $v$ such that ${\rm{alph}}(v) = [n] \setminus {\rm{alph}}(u)$ and ${\rm{red}}(v) = \tau$.

Consider two elements in ${\mathcal R}_n$, namely $w = u \cdot v$ and $w' = u' \cdot v'$. We will prove inductively on $n$ that if $w \neq w'$ then $w^{-1} \nsim_{ss} {w'}^{-1}$. To do that it suffices to show that the corresponding pyramidal sequences of consecutive differences $p(w)$ and $p(w')$ are not the same. We distinguish between two cases.

\begin{enumerate}[{Case }1:]
\item If $u \neq u'$, where $u, u' \in \mathcal{E}_{i,n}$, then we necessarily get $i\geq 2$ and hence $u$ and $u'$ are distinct elements in $\mathcal{D}_{i,n}$. In view of the bijection $\phi$ in Proposition \ref{prop:prefix<->trapezoid}, we obtain $\phi(u) \neq \phi(u')$ and therefore the corresponding pyramidal sequences are distinct.

\item If $u = u'$ then $v \neq v'$. On the other hand, ${\rm{alph}}(v) = {\rm{alph}}(v')$, which implies that ${\rm{red}}(v)$ and ${\rm{red}}(v')$ are distinct elements of $\mathcal{R}_{n-i}$. Our inductive argument on length immediately implies that $p({\rm{red}}(v)) \neq p({\rm{red}}(v'))$, hence the initial parts for $p(u\cdot v)$ and $p(u'\cdot v')$ are distinct since they are a scalar multiple of $p({\rm{red}}(v))$ and $ p({\rm{red}}(v'))$, respectively.
\end{enumerate}
In all cases, we get $p(u\cdot v) \neq p(u' \cdot v')$, as required.
\end{proof}

The recursively constructed sets $\mathcal{R}_{n}$, for $n \in [3, 6]$, are given in Table 5 (see Appendix), where all prefixes $u \in {\mathcal E}_{i,n}$ of permutations in $\mathcal{R}_{n}$ are highlighted in red colour.

\vspace{0.2in}

In our next result, we enumerate the set of super-stong Wilf equivalent classes of permutations of given cardinality. In view of Remark \ref{rem:2power}, such a cardinality is always a power of $2$. Let $s_{j,n}$ be the number of super-strong Wilf equivalence classes of order $2^j$ in $\mathcal{S}_n$, where $j \in [n-1]$. Note that $s_{0,n} = 0$. It is easy to check that $s_{1,2} = s_{1,3} = s_{2,3} = 1$.

\begin{theorem} \label{sjn}
For $n \geq 4$ we have
\[ s_{j,n} = \displaystyle  s_{j-1,n-1} + \sum_{k=2}^{n-j-1} d_{k,n} \cdot s_{j,n-k}. \]
\end{theorem}

\begin{proof}
Let ${\mathcal F}_{j,n}$ be the set of all pyramidal sequences in $\varPi_n$ with corresponding super-strong Wilf equivalence class of order $2^j$.
From \eqref{eq:partition} it clearly follows that
\begin{equation} \label{refinedpartition}
 {\mathcal F}_{j,n} = ({\mathcal F}_{j,n} \cap {\mathcal T}_{1,n}) \sqcup ({\mathcal F}_{j,n} \cap {\mathcal T}_{2,n})
\sqcup \cdots \sqcup  ({\mathcal F}_{j,n} \cap {\mathcal T}_{i,n}) \sqcup \cdots \sqcup ({\mathcal F}_{j,n} \cap {\mathcal T}_{n-2,n}). \end{equation}

By Remark \ref{rem:2power}, the exponent $j$ is equal to the number of transitions from $\Delta_k = (\underbrace{d, d, \ldots, d}_{n-k})$ to $\Delta_{k+1} = (\underbrace{d, d, \ldots, d}_{n-k-1})$, for $k\in [n-1]$.
Recall that for $k=n-1$, by convention $\Delta_{k+1}$ is the empty vector.

Let $i \in [2, n-1]$. Consider a pyramidal sequence $( \Delta_1, \ldots , \Delta_{n-1}) \in {\mathcal T}_{i,n}$. Observe that for $k \in [i]$ there are no transitions of the aforementioned form, therefore restricting the bijection ${\tau}_{i,n}$ in \eqref{eq:bijection_t} to $ {\mathcal F}_{j,n}$, the only contribution to the exponent $j$ comes from the part $( \Delta_{i+1}, \ldots , \Delta_{n-1})$. But the number of different classes there is equal to $s_{j, n-i}$. It follows that $|{\mathcal F}_{j,n} \cap {\mathcal T}_{i,n}| = d_{i,n}\cdot s_{j,n-i}$.

Now let $i=1$. Observe that there is a transition from ${\Delta}_1$ to $\Delta_2 = (\underbrace{1, 1, \ldots, 1}_{n-2})$ that raises the exponent of the order of the equivalence class by one, hence to get the desired exponent $j$ we need $j-1$ additional transitions on the upper part of the pyramid; there are precisely $s_{j-1, n-1}$ classes satisfying that and the result follows.
\end{proof}

\section{Numbers $d_{i,n}$ and non-interval permutations}
\label{sec:din}
The calculation of numbers $s_n$ and $s_{j,n}$ is based on the coefficients $d_{i,n}$. In this section we will present a recursive formula for the latter.

It turns out that numbers $d_{i,n}$ are related to the number $a_{n}$ of {\em non-interval} permutations. These are all permutations $s = s_1 s_2 \cdots s_n$ of size $n \geq 2$ such that the {\rm{alph}}abet of any prefix $s_1 \cdots s_l$ of length $2 \leq l < n$ is not equal to the interval $[k, l + k -1]$, for some $k$ (see \cite[Proposition 4.5]{aval}). The enumeration of non-interval permutations yields the integer sequence
\[ 2, \, 2, \, 8, \, 44, \, 296, \, 2312, \, 20384, \, 199376, \, 2138336, \, 24936416, \, 3141422848, \ldots \]
(also known as $|b_{n}|$, for $n \geq 2$, where $b_{n}$ is Sequence \underline{A077607} of \cite{oeis}) with recurrence formula
\begin{equation}
 \sum_{k=1}^{i} a_{k+1} \cdot {(i-k+1)}! = {(i+1)}!. \label{noninterval}
\end{equation}
Our initial calculations indicated the equation $d_{i, n} = a_{i+1}$, for relatively small $i$.
In fact, numbers $d_{i,n}$ satisfy the recursive formula of our following result, whose special case for $i < \lfloor \frac{n}{2} \rfloor$ coincides with \eqref{noninterval}.

Let $q_{l,m}$ and $r_{l,m}$ be the quotient and remainder, respectively, of an arbitrary integer $l$ with a non-zero integer $m$.

\begin{theorem} \label{numbersdin}
Let $n \geq 4$. For a given $i \in [n-2]$ set $m = n-i-1$. Then
\begin{equation} \label{eq: recdin}
 \sum_{k=1}^{i} \frac{q_{n-k,m}}{2} \cdot (r_{n-k,m} + i - k + 1) \cdot d_{k,n} \cdot {(i-k)}! =
                      \frac{q_{n,m}}{2} \cdot (r_{n,m} + i + 1) \cdot {i}!.
\end{equation}
\end{theorem}

For the proof of the theorem we will need the following result.

\begin{lemma} \label{numberspin}
Let $p_{i,n}$ be the number of all prefixes of length $i$ of permutations in $n$ letters with corresponding periodic suffix and set $m = n-i-1$. Then
\[ p_{i,n} = \frac{q_{n,m}}{2} \cdot (r_{n,m} + i + 1) \cdot {i}!. \]
\end{lemma}

\begin{proof}
It suffices to count the number $x_{i,n}$ of all ordered periodic words of length $n-i$ in $n$ distinct letters and then multiply it with the $i!$ choices for the remaining prefix letters. For a fixed $d \in [q_{n,m}]$, each such ordered periodic word has the form $a (a + d) \cdots (a + jd) \cdots (a + md)$, for $j \in [0,m]$ and is uniquely determined by the choice for the letter $a$. There exist precisely $n - dm$ such choices, so that \[ x_{i,n} = \sum_{d=1}^{q_{n,m}} (n- dm) = n q_{n,m} - m \sum_{d=1}^{q_{n,m}} d = n q_{n,m} - m \frac{q_{n,m}}{2} (q_{n,m} + 1). \]
Now since $n = m q_{n,m} + r_{n,m}$, we obtain $x_{i,n} = \frac{q_{n,m}}{2} (r_{n,m} + i +1)$ and our result follows.
\end{proof}

\begin{proof}[of Theorem \ref{numbersdin}]
We count the number of prefixes $p_{i,n}$ with periodic suffix in two ways. By Lemma \ref{numberspin} we directly obtain the number on the right hand side of \eqref{eq: recdin}.

An alternative counting method for $p_{i,n}$ is to start from the prefixes themselves. Consider a prefix $u$ of length $i$ such that the
remaining $(n-i)$-letter suffix is an ordered periodic word. Then there exists a unique $k \in [i]$ such that the prefix $u'$ of $u$ lies in ${\mathcal D}_{k,n}$. For this particular $k$-letter prefix $u'$, the number of all prefixes of length $i-k$ in the remaining $n-k$ distinct letters is equal to $p_{i-k, n-k}$ and by Lemma \ref{numberspin} we obtain that
\[ p_{i-k, n-k} = \frac{q_{n-k,m}}{2} (r_{n-k,m} + i - k +1)(i-k)!, \]
since $m = (n-k)-(i-k)-1 = n-i-1$. Note that for $k=i$ we get $p_{0,n-i}=1$, as expected.

A simple counting argument now yields
\[ p_{i,n} = \sum_{k=1}^{i} p_{i-k, n-k} \cdot d_{k,n} \]
and the result follows.
\end{proof}

\begin{remark}
Let $i < \lfloor \frac{n}{2} \rfloor$. Recall that $m = n - i -1$. Then it follows that $m \geq \lceil \frac{n}{2} \rceil - 1$, so that
$q_{n,m} = q_{n-k,m} = 1$ and $r_{n,m} = n - m$, $r_{n-k,m} = (n - k) - m$, for $k \in [i]$. Substituting in \eqref{eq: recdin} we obtain
\[ \sum_{k=1}^{i} d_{k, n} \cdot {(i - k + 1)}! = {(i + 1)}!. \]
It follows that $d_{k, n} = a_{k+1}$, for all $k < \lfloor \frac{n}{2} \rfloor$, since $d_{1, n} = a_{2} = 2$.
\end{remark}

This equality of cardinalities is not a mere coincidence. There is actually a deeper connection between ${\mathcal D}_{i,n}$ and the corresponding non-interval permutations.

\begin{proposition} \label{nonintervalbijection}
There is a bijection between the set of prefixes ${\mathcal D}_{k, n}$, for $k < \lfloor \frac{n}{2} \rfloor$ and the set ${\mathcal A}_{k+1}$
 of all non-interval permutations of length $k+1$.
\end{proposition}

\begin{proof}
It is more convenient to find a bijection from ${\mathcal D}_{k, n}$ to the set ${\mathcal B}_{k+1}$ of all permutations
$b = b_1 b_2 \cdots b_k b_{k+1}$ of size $k+1 \geq 2$ such that any suffix $b_{k-l+2} \cdots b_{k+1}$ of length $2 \leq l < k+1$ is not, up to order, equal to an interval of the form $[j, l + j - 1]$. This is possible due to the fact that ${\mathcal B}_{k+1}$ is clearly equinumerous to ${\mathcal A}_{k+1}$, via the bijection $w \mapsto \tilde{w}$.

Let $u = u_1 u_2 \cdots u_k$ be a prefix in ${\mathcal D}_{k, n}$. Let $a = \min ( [n] \setminus {\rm{alph}}(u) )$. We define a map $\rho : {\mathcal D}_{k, n} \longrightarrow {\mathcal B}_{k+1}$ as $\rho (u) = {\rm{red}}(u a)$.

For the reverse direction we define a map \\ $\theta : {\mathcal B}_{k+1} \longrightarrow {\mathcal D}_{k, n}$ as
$\theta (b_1 \cdots b_k b_{k+1}) = v_1 \cdots v_k$, where, for $i \in [k]$,
\begin{equation} \label{eq:psi}
v_i = \begin{cases}
b_i, \qquad \qquad \qquad \quad \, \mbox{if} \quad b_i < b_{k+1} \\
b_i + (n-k-1), \quad \mbox{if} \quad b_i > b_{k+1}.
\end{cases} \end{equation}
We begin by showing that maps $\rho$ and $\theta$ are well-defined.

To show that $\rho(u) \in {\mathcal B}_{k+1}$, for $u \in {\mathcal D}_{k,n}$, observe that since $k < \lfloor n/2 \rfloor$, the vector of consecutive differences of
$[n] \setminus {\rm{alph}}(u)$ is necessarily equal to $(\underbrace{1,1,\ldots,1}_{n-k-1})$,
hence \[\pi = u_1 \cdots u_k a (a+1) \cdots (a+n-k-1) \in {\mathcal S}_n.\] This implies that if we set $\rho(u) = b_1 \cdots b_k b_{k+1}$,
we must have
\begin{equation} \label{eq:prefinterv}
 b_i = \begin{cases}
u_i, \qquad \qquad \qquad \quad \, \mbox{if} \quad u_i < a \\
u_i - (n-k-1), \quad \mbox{if} \quad u_i > a,
\end{cases}
\end{equation}
for $i \in [k]$ and $b_{k+1} = a$.

Suppose, for the sake of contradiction, that $\rho(u) \notin {\mathcal B}_{k+1}$. Then consider the smallest index $j \in [k]$ such that
\[ \{b_j, \ldots , b_{k+1} \} = \{ a - l, \ldots , a - 1, a, a + 1, \ldots , a + r \}, \]
for some $l, r \in [k-j]$. Note that $l + r = k - j$. In view of \eqref{eq:prefinterv} we have that
\begin{equation} \label{eq:lr}
 \{u_j , \ldots , u_{k+1} \} = \{ a - l, \ldots , a - 1, a , a + (n-k), \ldots , a + r + (n-k-1) \}.
\end{equation}
Combining the form of the permutation $\pi$ and \eqref{eq:lr}, we obtain that $u_1 \cdots u_{j-1} $ lies in $ {\mathcal D}_{j-1, n}$, which contradicts the definition of ${\mathcal D}_{k, n}$.

Let $b=b_1 \cdots b_k b_{k+1} \in \mathcal{B}_{k+1}$. We need to show that $v =v_1 \cdots v_k = \theta(b) \in \mathcal{D}_{k,n}$. By definition of $\theta$, it follows that
\[ v_1 \cdots v_k \ b_{k+1} \ (b_{k+1} + 1) \ \cdots  (b_{k+1} + n-k-1) \in \mathcal{S}_n,\]
and consequently a prefix of $v$ (not necessarily proper) lies in $\mathcal{D}_{j,n}$ for some $j\leq k$. If $j<k$, then $b$ would have a suffix of length $k+1-j$ corresponding to an interval, a contradiction.

It remains to show that $\rho\circ\theta = id_{\mathcal{B}_{k+1}}$ and $\theta\circ\rho = id_{\mathcal{D}_{k,n}}$. We will demonstrate the former, the latter follows using similar arguments. Having that $b = b_1 \cdots b_k b_{k+1}\in \mathcal{B}_{k+1}$ and in view of \eqref{eq:psi}, we obtain ${\rm{alph}}(v_1 \cdots v_k) = [b_{k+1} - 1] \cup [b_{k+1} + n-k, n]$. Hence, $\min\big([n] \setminus {\rm{alph}}(v) \big) = b_{k+1}$ and $\rho(v) = {\rm{red}}(v  b_{k+1}) = b$, as required.
\end{proof}

\begin{example}
Let us calculate the number $s_{10}$ of super-strong Wilf equivalence classes of ${\mathcal S}_{10}$. In view of Theorem \ref{sncounting}, we need to calculate the numbers $d_{i, 10}$, for $i \in [8]$. Since $\lfloor 10/2 \rfloor = 5$, Proposition \ref{nonintervalbijection} immediately yields
$d_{1, 10} = a_{2} = 2$, $d_{2, 10} = a_{3} = 2$, $d_{3, 10} = a_{4} = 8$ and $d_{4, 10} = a_{5} = 44$.

To evaluate $d_{5, 10}$ set $i = 5$ and $m = 4$ in \eqref{eq: recdin} and since
$q_{10,4} = q_{9,4} = q_{8,4} = 2$; $q_{7,4} = q_{6,4} = 1$; $r_{10,4} = r_{6,4} = 2$,
$r_{9,4} = 1$, $r_{8,4} = 0$ and $r_{7,4} = 3$, we obtain
\[ \frac{2}{2} \cdot (1 + 5)\cdot 2  \cdot 4! + \frac{2}{2} \cdot (0 + 4) \cdot 2 \cdot 3! + \frac{1}{2} \cdot (3 + 3) \cdot 8 \cdot 2!
   + \frac{1}{2} \cdot (2 + 2) \cdot 44 \cdot 1!  + d_{5, 10} \! = \! \frac{2}{2} \cdot (2 + 6) \cdot 5!. \]
It follows that $d_{5,10} = 488$. Using \eqref{eq: recdin} in a similar fashion, we get $d_{6,10} = 4,664$, $d_{7,10} = 58,336$ and $d_{8,10} = 1,114,944$.

Substituting the above values of $d_{i,10}, \, i \in [2, 8]$ and the recursively calculated values $s_2 = 1$, $s_3 = 2$, $s_4 = 8$, $s_5 = 40$, $s_6 = 256$, $s_7 = 1,860$, $s_8 = 15,580$ and $s_9 = 144,812$, for $i < 10$ to formula \eqref{eq: sncount}, we finally obtain $s_{10} = 1,490,564$.
\end{example}

Using the same reasoning we have calculated all numbers $s_n$, for $n \in [12]$ (see Table 2 in the Appendix). The calculation is based on the numbers $d_{i,n}$, for $i \in [10]$ and $n \in [3, 12]$ (see Table 1 in the Appendix). In particular, in view of Proposition \ref{nonintervalbijection}, numbers $d_{k, n}$, for all $k < \lfloor \frac{n}{2} \rfloor$, i.e., numbers of all non-interval permutations of length $k+1$, are highlighted with {\rm{red}} colour in that table. Finally, numbers $s_{j, n}$, for $i \in [11]$ and $n \in [2, 12]$, are shown in Table 4 in Appendix.

\section{Shift and super-strong Wilf equivalence classes of permutations}
\label{sec:shift}
{\em Shift equivalence} is a geometric notion which was firstly introduced for words in general in \cite{FGMPX} providing a new insight in Wilf equivalence theory. It is defined using the notions of the {\em skyline diagram} and {\em rigid shift} of a word. The skyline diagram of a word $u = u_1 u_2 \cdots u_n \in \mathbb{P}^{*}$ of length $n$ is the geometric figure formed by adjoining $n$ columns of squares such that the $i$-th column is made up of $u_i$ squares. A rigid shift of a word $u$ is any word $v$ that can be formed by cutting the skyline diagram of $u$ at some height $h$ and rigidly moving together all blocks above the cut line in such a way that each moved column comes to rest on a column of height $h$.

\begin{definition}\cite[p. 4]{FGMPX}
Two words are \emph{shift equivalent} if one can be obtained from the other by performing any finite sequence of rigid shifts and reversals.
\end{definition}

Restricting ourselves to permutations we define the following notion.

\begin{definition}
Two permutations are \emph{strongly shift equivalent} if one can be obtained from the other by performing any finite sequence of rigid shifts only.
\end{definition}

\begin{example}
Consider the permutation $u = 32415 \in {\mathcal S}_{5}$. Then $v = 42513$ is a rigid shift of $u$ because $v$ is constructed by cutting the skyline diagram of $u$ at height $3$ and rigidly moving all blocks above it by two positions to the left. The permutation $w = 31524$ can not be obtained from $u$ by a sequence of rigid shifts. Nevertheless, $w$ is shift equivalent to $u$ since $\tilde{w} = v$.
\end{example}

We have the following result that relates super-strong Wilf equivalence and strong shift equivalence for permutations and provides us with an alternate way to characterize the former.

\begin{proposition} \label{ssrigid}
Let $u, v\in \mathcal{S}_n$. Then $u$ is strongly shift equivalent to $v$ if and only if $u\sim_{ss} v$.
\end{proposition}

\begin{proof}
Suppose that $u$ is strongly shift equivalent to $v$. Since $u$ is obtained from $v$ by a finite sequence of rigid shifts, it suffices to show that if $v$ is obtained from $u$ by performing a single rigid shift, then $u\sim_{ss} v$. The proof of \cite[Theorem 1]{FGMPX} is based on the construction of a weight-preserving bijection from the $m$-minimal clusters of $u$ to the $m$-minimal clusters of $v$. It is shown there that, for a fixed embedding set $E$, $m(v,E)$ is itself a rigid shift of $m(u,E)$, hence it is also a rearrangement of $m(u,E)$, for all embedding index sets $E$. Using the {\em{Minimal Cluster Rearrangement Theorem}}, a necessary and sufficient criterion for super-strong Wilf equivalence (see \cite[Theorem 4]{HMS}), we conclude that
$u \sim_{ss} v$.

For the converse, suppose that $u\sim_{ss} v$. This is equivalent to $\Delta_i(s) = \Delta_i(t) = \Delta_i$ for all $i\in [n-1]$, where $s=u^{-1}$ and $t=v^{-1}$. We will construct a sequence of rigid shifts that transforms $u$ to $v$, starting from greater letters to smaller ones. This is done inductively for $i$ from $n-1$ down to $1$.

Let $i\in [n-1]$. Its relative position with respect to the greatest letters in $u$ and $v$ depends only on the transition from $\Delta_{i+1}$ to $\Delta_i$. The only case where we have more than one choice for this placement is when $\Delta_{i+1} = (\underbrace{d, d, \ldots, d}_{n-i-1})$ and $\Delta_{i} = (\underbrace{d, d, \ldots, d}_{n-i})$. In this case, the factors of the words $u$ and $v$ that correspond to the previous distance vector $\Delta_{i+1}$ are strongly shift equivalent by induction, and may be written in the form
\[
* \ \underbrace{\circ \cdots \circ}_{d-1} \ * \ \underbrace{\circ \cdots \circ}_{d-1} \ * \ \cdots \ * \ \underbrace{\circ \cdots \circ}_{d-1} \ * \ \underbrace{\circ \cdots \circ}_{d-1} *,
\]
where $*$ corresponds to letters greater than $i$ and $\circ$ corresponds to letters less than or equal to $i$. There are two choices for the placement of the letter $i$ on $u$ and $v$. We may assume that the character $i$ is placed on the left of the above configuration for $u$ as follows
\[
\fbox{i} \ \underbrace{\circ \cdots \circ}_{d-1} \ * \ \underbrace{\circ \cdots \circ}_{d-1} \ * \ \cdots \ * \ \underbrace{\circ \cdots \circ}_{d-1} \ * \ \underbrace{\circ \cdots \circ}_{d-1} *.
\]
If $i$ is also placed on the left of the above configuration for $v$, no rigid shift is needed. On the other hand, if it is placed on the right, then by applying \emph{a rigid shift of height $i$ by $d$ places} to the left, we get the same configuration, or a strongly shift equivalent configuration inherited by the previous steps as in $u$.

In the case where $i$ was placed originally to the right of the above configuration for $u$, the rigid shift, if needed, would be of the same type to the right.
\end{proof}

We are now ready to describe the shift equivalence class ${[u]}_{sh}$ of a permutation $u$ in terms of the corresponding super-strong Wilf classes ${[u]}_{ss}$ and ${[ {\tilde{u}}]}_{ss}$ of $u$ and its reversal $\tilde{u}$, respectively.

\begin{theorem} \label{shiftclass}
Let $v, id \in {\mathcal S}_{n}$, where $v = 1 \ 2 \ \cdots \ (n-3) \ (n-1) \ (n-2) \ n$ and $id$ is the identity permutation.
Then for an arbitrary element $u \in {\mathcal S}_{n}$ we have
\[
{[u]}_{sh} = \begin{cases}
{[u]}_{ss} , \qquad \qquad \text{if } \ u \, {\sim}_{ss} \, id, \: \mbox{or} \: \, u \, {\sim}_{ss} \, v, \\
{[u]}_{ss} \, \sqcup \, {[ {\tilde{u}}]}_{ss}, \: \, \, \mbox{otherwise}.
\end{cases}
\]
\end{theorem}

\begin{proof}
By \cite[Theorem 22]{HMS} it is known that the only super-strong Wilf equivalence classes that remain invariant under the application of the reversal map $u \mapsto \tilde{u}$ in ${\mathcal S}_{n}$ are the classes ${\mathcal I}_{n} = {[id]}_{ss}$ and ${\mathcal M}_{n} = {[v]}_{ss}$. Note that in terms of pyramidal sequences, the classes ${\mathcal I}_{n}$ and ${\mathcal M}_{n}$ correspond to the pyramid that contains only $1$'s and the pyramid with $1$'s everywhere except from the uppermost part which is ${\Delta}_{n-1} = (2)$, respectively. Our result now follows directly from Proposition \ref{ssrigid} and the definition of shift equivalence.
\end{proof}

The enumeration of shift equivalence classes of permutations of size $n$ is given by the next result which follows immediately by the previous theorem.

\begin{corollary} \label{shiftenum}
Let $sh_n$ be the number of distinct shift equivalence classes of the symmetric group ${\mathcal S}_{n}$ for $n \geq 3$. Then
\[ sh_n = 1 + \frac{s_n}{2}. \]
\end{corollary}

Numbers $sh_n$, for $n \in [12]$ are shown in Table 3 (see Appendix).

\section{Conclusions and further research}
The main result of this paper is a recursive formula for the number $s_n$ of super-strong Wilf equivalence classes of permutations in $n$ letters. This has been done by introducing a new family of numbers $d_{i,n}$ which enumerate either trapezoidal sequences of height $i$ in pyramidal sequences of consecutive differences for permutations of size $n$, or prefixes of length $i$ of some sort of generalized non-interval permutations of size $n$.

One problem that has not been dealt here is the evaluation of a generating function for the sequence $s_n$. In view of Theorem \ref{sncounting}, this would require a bivariate generating function for the double indexed sequence $d_{i,n}$, a problem that might be of interest on its own as these numbers could  appear elsewhere. To this purpose, one needs to treat formulae \eqref{eq: recdin} and \eqref{eq: sncount} with some exceptional care, as variable $n$ in \eqref{eq: sncount} is involved in both numbers $d_{i,n}$ and $s_n$. This project seems to be quite ambitious.

\acknowledgements
\label{sec:ack}
We would like to thank the anonymous referees for their valuable comments that improved this article. The first author would also like to dedicate this work to the memory of Athanassios Basogiannis, a truly inspirational mentor and teacher of mathematics from his high school years, who recently passed away.

\nocite{*}
\bibliographystyle{abbrvnat}
\bibliography{ssWilf-dmtcs}
\label{sec:biblio}

\newpage
\vspace{-2cm}
\begin{appendices}
\section*{Appendix}

\begin{tabular}{| r||c|c|c|c|c|c|c|c|c|c| } \hline
$i \backslash n$ & 3 & 4 & 5 & 6 & 7 & 8 & 9 & 10 & 11 & 12 \\
\hline \hline
1    & 3 & {\color{red}{2}}  & {\color{red}{2}}  & {\color{red}{2}}   & {\color{red}{2}}  & {\color{red}{2}}  & {\color{red}{2}}  & {\color{red}{2}}  & {\color{red}{2}}  & {\color{red}{2}}  \\
\hline
2     &  & 6 & 4  &  {\color{red}{2}} & {\color{red}{2}} & {\color{red}{2}} & {\color{red}{2}}  & {\color{red}{2}}  & {\color{red}{2}} & {\color{red}{2}}  \\
\hline
3     &  &  & 24 & 16 & 14  & {\color{red}{8}}  & {\color{red}{8}}  & {\color{red}{8}}  & {\color{red}{8}}  & {\color{red}{8}}  \\
\hline
4     &  &  &  & 168 & 100  & 80  & 68  & {\color{red}{44}}  & {\color{red}{44}}  & {\color{red}{44}} \\
\hline
5     &  & &  &  & 1,212  & 712  & 500  & 488  & 416  & {\color{red}{296}} \\
\hline
6     &  &  &  &  &    & 10,824  & 6,376  & 4,664  &  3,704 & 3,512 \\
\hline
7     &  &  &  &  &   &   & 103,992  & 58,336  & 43,592  & 33,152 \\
\hline
8     &  &  &  &  &   &   &   & 1,114,944  & 630,544  & 444,992 \\
\hline
9     &  &  &  &  &   &   &   &   & 12,907,824  & 7,167,802 \\
\hline
10     &  &  &  &  &   &   &   &   &   & 162,773,970 \\
\hline
\end{tabular}
\begin{center}
Table 1: The numbers $d_{i,n}$, for $1 \leq i \leq 10$ and $3 \leq n \leq 12$
\end{center}

\vspace{0.5cm}

\noindent \begin{tabular}
{| r||c|c|c|c|c|c|c|c|c|c|c|c| } \hline
$n$ & 1 & 2 & 3 & 4 & 5 & 6 & 7 & 8 & 9 & 10 & 11 & 12 \\
\hline \hline
$s_n$     & 1 & 1  & 2  & 8  & 40  & 256  & 1,860  & 15,580  & 144,812  & 1,490,564 & 16,758,972 & 205,029,338 \\
\hline
\end{tabular}
\begin{center}
Table 2: The numbers $s_n$, for $1 \leq n \leq 12$
\end{center}

\vspace{0.5cm}

\noindent \begin{tabular}
{| r||c|c|c|c|c|c|c|c|c|c|c|c| } \hline
$n$ & 1 & 2 & 3 & 4 & 5 & 6 & 7 & 8 & 9 & 10 & 11 & 12 \\
\hline \hline
$sh_n$     & 1 & 1  & 2  & 5  & 21  & 129  & 931  & 7,791  & 72,407  & 745,283 & 8,379,487 & 102,514,670 \\
\hline
\end{tabular}
\begin{center}
Table 3: The numbers $sh_n$, for $1 \leq n \leq 12$
\end{center}

\vspace{0.5cm}

\noindent \begin{tabular}{| r||c|c|c|c|c|c|c|c|c|c|c| } \hline
$j \backslash n$ & 2 & 3 & 4 & 5 & 6 & 7 & 8 & 9 & 10 & 11 & 12 \\
\hline \hline
1     & 1 & 1 & 6  & 28  & 196  & 1,452  & 12,632  & 119,744  & 1,260,432  & 14,389,600 & 178,692,748 \\
\hline
2     &  & 1 & 1 & 10  & 46  & 330  & 2,416  & 21,216  & 197,120  & 2,067,024 & 23,263,418 \\
\hline
3     &   &  & 1 & 1 & 12  & 62  & 442  & 3,204  & 28,276  & 262,080 & 2,707,296 \\
\hline
4     &  & &  & 1 & 1  &  14 & 72  &  546 & 3,992  & 34,680 & 318,408 \\
\hline
5     &   &  &  &  &  1 &  1 & 16  & 82  & 630  & 4,744 & 41,108 \\
\hline
6     &  &  &  &  &    & 1  &  1 & 18  & 92  & 718 & 5,412 \\
\hline
7     &  &  &  &  &   &  & 1  &  1 &  20 & 102 & 810 \\
\hline
8     &  &  &  &  &  &   &   &  1 &  1 & 22 & 112 \\
\hline
9     &  &  &  &  &   &   &   &  &  1 & 1 & 24 \\
\hline
10     &  &  &  &  &   &   &   &   &  & 1 & 1 \\
\hline
11     &  &  &  &  &   &   &  &   &   &  & 1 \\
\hline
\end{tabular}
\begin{center}
Table 4: The numbers $s_{j,n}$, for $1 \leq j \leq 11$ and $2 \leq n \leq 12$
\end{center}

\vspace{-5cm}

\noindent \begin{tabular}{|c||l|} \hline
$n$ & $\mathcal{R}_{n}$ \\
\hline \hline
3 & {\color{red}{1}}23, {\color{red}{2}}13 \\ \hline
4 & {\color{red}{1}}234, {\color{red}{1}}324 \\ \cline{2-2}
 & {\color{red}{21}}34, {\color{red}{23}}14, {\color{red}{24}}13, {\color{red}{31}}24, {\color{red}{32}}14, {\color{red}{34}}12 \\ \hline
5 & {\color{red}{1}}2345, {\color{red}{1}}2435, {\color{red}{1}}3245, {\color{red}{1}}3425, {\color{red}{1}}3524, {\color{red}{1}}4235, {\color{red}{1}}4325, {\color{red}{1}}4523; \\ \cline{2-2}
 & {\color{red}{21}}345, {\color{red}{21}}435; {\color{red}{24}}135, {\color{red}{24}}315; {\color{red}{42}}135, {\color{red}{42}}315; {\color{red}{45}}123, {\color{red}{45}}213; \\ \cline{2-2}
  & {\color{red}{231}}45, {\color{red}{234}}15, {\color{red}{235}}14, {\color{red}{251}}34, {\color{red}{253}}14, {\color{red}{254}}13, {\color{red}{312}}45, {\color{red}{314}}25, {\color{red}{315}}24, {\color{red}{321}}45, {\color{red}{324}}15, {\color{red}{325}}14, \\
  & {\color{red}{341}}25, {\color{red}{342}}15, {\color{red}{345}}12, {\color{red}{351}}24, {\color{red}{352}}14, {\color{red}{354}}12, {\color{red}{412}}35, {\color{red}{413}}25, {\color{red}{415}}23, {\color{red}{431}}25, {\color{red}{432}}15, {\color{red}{435}}12 \\ \hline
6 &  {\color{red}{1}}23456, {\color{red}{1}}23546, {\color{red}{1}}24356, {\color{red}{1}}24536, {\color{red}{1}}24635, {\color{red}{1}}25346, {\color{red}{1}}25436, {\color{red}{1}}25634, \\
 & {\color{red}{1}}32456, {\color{red}{1}}32546, {\color{red}{1}}35246, {\color{red}{1}}35426, {\color{red}{1}}53246, {\color{red}{1}}53426, {\color{red}{1}}56234, {\color{red}{1}}56324,  \\
 & {\color{red}{1}}34256, {\color{red}{1}}34526, {\color{red}{1}}34625, {\color{red}{1}}36245, {\color{red}{1}}36425, {\color{red}{1}}36524, {\color{red}{1}}42356, {\color{red}{1}}42536, {\color{red}{1}}42635, {\color{red}{1}}43256, \\
 & {\color{red}{1}}43526, {\color{red}{1}}43625, {\color{red}{1}}45236, {\color{red}{1}}45326, {\color{red}{1}}45623, {\color{red}{1}}46235, {\color{red}{1}}46325, {\color{red}{1}}46523, {\color{red}{1}}52346, {\color{red}{1}}52436, \\
 & {\color{red}{1}}52634, {\color{red}{1}}54236, {\color{red}{1}}54326, {\color{red}{1}}54623; \\ \cline{2-2}
 & {\color{red}{21}}3456, {\color{red}{21}}3546, {\color{red}{21}}4356, {\color{red}{21}}4536, {\color{red}{21}}4635, {\color{red}{21}}5346, {\color{red}{21}}5436, {\color{red}{21}}5634; \\
 & {\color{red}{56}}1234, {\color{red}{56}}1324, {\color{red}{56}}2134, {\color{red}{56}}2314, {\color{red}{56}}2413, {\color{red}{56}}3124, {\color{red}{56}}3214, {\color{red}{56}}3412; \\ \cline{2-2}
 & {\color{red}{231}}456, {\color{red}{231}}546; {\color{red}{246}}135, {\color{red}{246}}315; {\color{red}{261}}345, {\color{red}{261}}435;
 {\color{red}{264}}135, {\color{red}{264}}315; {\color{red}{312}}456, {\color{red}{312}}546; \\
 & {\color{red}{315}}246, {\color{red}{315}}426; {\color{red}{321}}456, {\color{red}{321}}546; {\color{red}{351}}246, {\color{red}{351}}426;
   {\color{red}{426}}135, {\color{red}{426}}315; {\color{red}{456}}123, {\color{red}{456}}213; \\
 & {\color{red}{462}}135, {\color{red}{462}}315; {\color{red}{465}}123, {\color{red}{465}}213; {\color{red}{513}}246, {\color{red}{513}}426;
   {\color{red}{516}}234, {\color{red}{516}}324; {\color{red}{531}}246, {\color{red}{531}}426; \\
 & {\color{red}{546}}123, {\color{red}{546}}213; \\ \cline{2-2}
 & {\color{red}{2341}}56, {\color{red}{2345}}16, {\color{red}{2346}}15; {\color{red}{2351}}46, {\color{red}{2354}}16, {\color{red}{2356}}14;
   {\color{red}{2361}}45, {\color{red}{2364}}15, {\color{red}{2365}}14; \\
 & {\color{red}{2413}}56, {\color{red}{2415}}36, {\color{red}{2416}}35; {\color{red}{2431}}56, {\color{red}{2435}}16, {\color{red}{2436}}15;
  {\color{red}{2451}}36, {\color{red}{2453}}16, {\color{red}{2456}}13; \\
 & {\color{red}{2513}}46, {\color{red}{2514}}36, {\color{red}{2516}}34; {\color{red}{2531}}46, {\color{red}{2534}}16, {\color{red}{2536}}14;
  {\color{red}{2541}}36, {\color{red}{2543}}16, {\color{red}{2546}}13; \\
 & {\color{red}{2561}}34, {\color{red}{2563}}14, {\color{red}{2564}}13; {\color{red}{2631}}45, {\color{red}{2634}}15, {\color{red}{2635}}14;
  {\color{red}{2651}}34, {\color{red}{2653}}14, {\color{red}{2654}}13; \\
 & {\color{red}{3142}}56, {\color{red}{3145}}26, {\color{red}{3146}}25; {\color{red}{3162}}45, {\color{red}{3164}}25, {\color{red}{3165}}24;
  {\color{red}{3241}}56, {\color{red}{3245}}16, {\color{red}{3246}}15; \\
 & {\color{red}{3251}}46, {\color{red}{3254}}16, {\color{red}{3256}}14; {\color{red}{3261}}54, {\color{red}{3264}}15, {\color{red}{3265}}14;
  {\color{red}{3412}}56, {\color{red}{3415}}26, {\color{red}{3416}}25; \\
 & {\color{red}{3421}}56, {\color{red}{3425}}16, {\color{red}{3426}}15; {\color{red}{3451}}26, {\color{red}{3452}}16, {\color{red}{3456}}12;
  {\color{red}{3461}}25, {\color{red}{3462}}15, {\color{red}{3465}}12; \\
 & {\color{red}{3521}}46, {\color{red}{3524}}16, {\color{red}{3526}}14; {\color{red}{3541}}26, {\color{red}{3542}}16, {\color{red}{3546}}12;
  {\color{red}{3561}}24, {\color{red}{3562}}14, {\color{red}{3564}}12; \\
 & {\color{red}{3612}}45, {\color{red}{3614}}25, {\color{red}{3615}}24; {\color{red}{3621}}45, {\color{red}{3624}}15, {\color{red}{3625}}14;
  {\color{red}{3641}}25, {\color{red}{3642}}15, {\color{red}{3645}}12;  \\
& {\color{red}{3651}}24, {\color{red}{3652}}14, {\color{red}{3654}}12; {\color{red}{4123}}56, {\color{red}{4125}}36, {\color{red}{4126}}35;
   {\color{red}{4132}}56, {\color{red}{4135}}26, {\color{red}{4136}}25;  \\
 & {\color{red}{4152}}36, {\color{red}{4153}}26, {\color{red}{4156}}23; {\color{red}{4162}}35, {\color{red}{4163}}25, {\color{red}{4165}}23;
  {\color{red}{4213}}56, {\color{red}{4215}}36, {\color{red}{4216}}35; \\
 & {\color{red}{4231}}56, {\color{red}{4235}}16, {\color{red}{4236}}15; {\color{red}{4251}}36, {\color{red}{4253}}16, {\color{red}{4256}}13;
  {\color{red}{4312}}56, {\color{red}{4315}}26, {\color{red}{4316}}25; \\
 & {\color{red}{4321}}56, {\color{red}{4325}}16, {\color{red}{4326}}15; {\color{red}{4351}}26, {\color{red}{4352}}16, {\color{red}{4356}}12;
  {\color{red}{4361}}25, {\color{red}{4362}}15, {\color{red}{4365}}12; \\
 & {\color{red}{4512}}36, {\color{red}{4513}}26, {\color{red}{4516}}23; {\color{red}{4521}}36, {\color{red}{4523}}16, {\color{red}{4526}}13;
  {\color{red}{4531}}26, {\color{red}{4532}}16, {\color{red}{4536}}12; \\
 & {\color{red}{4612}}35, {\color{red}{4613}}25, {\color{red}{4615}}23; {\color{red}{4631}}25, {\color{red}{4632}}15, {\color{red}{4635}}12;
  {\color{red}{5123}}46, {\color{red}{5124}}36, {\color{red}{5126}}34; \\
 & {\color{red}{5142}}36, {\color{red}{5143}}26, {\color{red}{5146}}23; {\color{red}{5213}}46, {\color{red}{5214}}36, {\color{red}{5216}}34;
  {\color{red}{5231}}46, {\color{red}{5234}}16, {\color{red}{5236}}14; \\
 & {\color{red}{5241}}36, {\color{red}{5243}}16, {\color{red}{5246}}13; {\color{red}{5261}}34, {\color{red}{5263}}14, {\color{red}{5264}}13;
  {\color{red}{5321}}46, {\color{red}{5324}}16, {\color{red}{5326}}14; \\
 & {\color{red}{5341}}26, {\color{red}{5342}}16, {\color{red}{5346}}12; {\color{red}{5361}}24, {\color{red}{5362}}14, {\color{red}{5364}}12;
  {\color{red}{5412}}36, {\color{red}{5413}}26, {\color{red}{5416}}23; \\
 & {\color{red}{5421}}36, {\color{red}{5423}}16, {\color{red}{5426}}13;  {\color{red}{5431}}26, {\color{red}{5432}}16, {\color{red}{5436}}12; \\
\hline
\end{tabular}

\vspace{0.3cm}

\begin{center}
Table 5: The sets $\mathcal{R}_n$, for $3 \leq n \leq 6$
\end{center}

\end{appendices}

\end{document}